\documentclass[10pt,oneside]{amsart}
\usepackage{amssymb,amscd}
\usepackage{amsmath}
\usepackage{graphicx}

\let\phi\varphi
\let\kappa\varkappa
\let\le\leqslant
\let\ge\geqslant

\let\emptyset\varnothing

\let\ES\varnothing

\newcommand\subcl{\mathrel{\underset{\mathrm{cl}}{\subset}}}

\newcommand\Cl{\mathop{\mathrm{Cl}}}

\newcommand{\pr}{\mathop{\mathrm{pr}}\nolimits}

\newcommand{\uni}[1]{\mathbf{1}_{{#1}}}

\newcommand\tle{\mathrel{\tilde{\le}}}

\newcommand\ups{{\uparrow}}
\newcommand\dns{{\downarrow}}
\newcommand\ddns{{\underset{\raisebox{9.5pt}[0pt][0pt]{$\downarrow$}}{\downarrow}}}

\newcommand\Semz{\mathcal{S}\mathrm{em}_{0}}

\newcommand\Sempr{\mathcal{S}\mathrm{em}\mathcal{PR}}
\newcommand\Sempral{\mathcal{S}\mathrm{em}\mathcal{PR}^*_L}
\newcommand\Semprhl{\mathcal{S}\mathrm{em}\mathcal{PR}^{\hat *}_L}
\newcommand\hast{\mathbin{\hat *}}
\newcommand\hoast{\mathbin{\hat\circledast}}
\newcommand\astsc{\mathpunct{\overset{{\,*}}{,}\hskip0pt}}
\newcommand\hastsc{\mathpunct{\overset{{\,\hat *}}{,}\hskip0pt}}
\let\sms\smallsmile
\newcommand\ssms{{%
\rlap{$\scriptstyle\smallsmile$}%
\raise .5ex \hbox{$\scriptstyle\smallsmile$}%
}}

\newcommand{\BBN}{\mathbb{N}}

\newcommand{\CCA}{\mathcal{A}}

\usepackage{xy}
\xyoption{all}

\makeatother
\swapnumbers

\newtheorem*{theorem*}{Theorem}
\newtheorem{lemma}[subsection]{Lemma}
\newtheorem{proposition}[subsection]{Proposition}
\newtheorem{corollary}[subsection]{Corollary}
\theoremstyle{definition}
\newtheorem{definition}[subsection]{Definition}
\newtheorem*{definition*}{Definition}

\newtheorem*{remark*}{Remark}
\newtheorem{example}[subsection]{Example}
\newtheorem*{example*}{Example}

\newtheorem*{question*}{Question}
\begin{document}

\title{Ambiguous representations of semilattices and imperfect information}

\author{Oleh Nykyforchyn, Oksana Mykytsey 
}

\address{ Casimirus the Great University in Bydgoszcz\\
    Institute of Mathematics\\
    12 Weyssenhoff Sq., Bydgoszcz, Poland, 85064
}

\email{oleh.nyk@gmail.com}

\address{Vasyl' Stefanyk Precarpathian National University \\
    Department of Mathematics and Computer Science \\
    57 Shevchenka St., Ivano-Frankivsk, Ukraine, 76025\\
}

\email{oxana37@i.ua}


\keywords{Ambiguous representation, continuous semilattice, duality of categories, imperfect information}
\subjclass{06B35;  
           18B30} 

\maketitle

\begin{abstract}
Crisp and lattice-valued ambiguous representations of one
continuous semilattice in another one are introduced and operation
of taking pseudo-inverse of the above relations is defined. It is
shown that continuous semilattices and their ambiguous
representations, for which taking pseudo-inverse is involutive,
form categories. Self-dualities and contravariant equivalences for
these categories are obtained. Possible interpretations and
applications to processing of imperfect information are discussed.
\end{abstract}

\section*{Introduction}

The goal of this work is to generalize notions of crisp and
$L$-fuzzy ambiguous representations introduced in
\cite{NykRep:AmbRep:11} for closed sets in compact Hausdorff
spaces. Fuzziness and roughness were combined to express the~main
idea that a~set in one space can be represented with a~set in
another space, e.g., a 2D photo can represent 3D object. This
representation is not necessarily unique, and the~object cannot be
recovered uniquely, hence we say ``ambiguous representation''.

It turned out that most of results of \cite{NykRep:AmbRep:11} can
be extended to wider settings, namely to continuous semilattices,
which are standard tool to represent partial information. This is
done in the~present paper, and an~interpretation of the~considered
objects is proposed.

The paper is organized as follows. First necessary definitions and
facts are given on continuous (semi-)lattices. Then we define
compatibilities, which are functions with values 0 and 1 that show
whether two pieces of information can be valid together. Next
ambiguous representations are introduced as crisp and $L$-fuzzy
relations between continuous semilattices. Operation of taking
pseudo-inverse is defined for these relations, its properties are
proved, and classes of pseudo-invertible representations are
investigated. It is shown that continuous semilattices and their
pseudo-invertible crisp and $L$-fuzzy ambiguous representations
form categories, and self-dualities and contravariant equivalences
are construsted for these categories.
Finally, we discuss possible meaning of
the~developed theory. Reader can start with the last section to get
the~general idea what is in the~previous parts of the~paper.

\section{Preliminaries}

We adopt the following definitions and notation, which are
consistent with \cite{GHK:CLD:03,ML:CWM:98}. Proofs of the facts
below can also be found there. From now on,
\emph{semilattice} means \emph{meet semilattice}, if otherwise is
not specified. If a~poset contains a~bottom (a~top) element, then
it is denoted by $0$ (resp.\ by $1$). A~top (a~bottom) element in
a~semilattice is also called a~\emph{unit} (resp.\ a~\emph{zero}).

For a~partial order $\le$ on a~set $X$, the~relation $\tle$,
defined as $x\tle y\iff y\le x$, for $x,y\in X$, is a partial order
called \emph{opposite} to $\le$, and $(X,\le)^{op}$ denotes the
poset $(X,\tle)$. If the original order $\le$ is obvious, we write
simply $X^{op}$ for the \emph{reversed} poset. We also apply
$\tilde{(\;)}$ to all notation to denote passing to the~opposite
order, i.e.\ write $\tilde X=X^{op}$, $\tilde\sup=\inf$, $\tilde
0=1$ etc. For a~morphism $f:(X\le)\to (Y,\le)$ in a~category
$\mathcal{P}\mathrm{oset}$ of posets and isotone (order preserving)
mappings, let $f^{op}$ be the~same mapping, but regarded as
$(X,\tilde \le)\to (Y,\tilde
\le)$. It is obvious that $f^{op}$ is isotone as well, thus
a~functor ${(-)}^{op}:\mathcal{P}\mathrm{oset}\to
\mathcal{P}\mathrm{oset}$ is obtained.

For a~subset $A$ of a poset $(X,\le)$, we denote
$$
A\ups=\{x\in X\mid a\le x\text{ for some }a\in A\},\;
A\dns=\{x\in X\mid x\le a\text{ for some }a\in A\}.
$$
If $A=A\ups$ ($A=A\dns$), then a~set $A$ is called \emph{upper}
(resp.\ \emph{lower}).

A~\emph{topological meet} (or \emph{join}) \emph{semilattice} is a
semilattice $L$ carrying a topology such that the mapping $\land:L
\times L\to L$ (resp.\ $\lor:L\times L\to L$) is continuous.
A~lattice $L$ with a~topology such that both $\land:L
\times L\to L$ and $\lor:L\times L\to L$ are continuous is called
a~\emph{topological lattice}.

A~poset $(X,\le)$ is called \emph{complete} if each non-empty
subset $A\subset X$ has a~least upper bound.

A set $A$ in a~poset $(X,\le)$ is \emph{directed} (\emph{filtered})
if, for all $x,y\in A$, there is $z\in A$ such that $x\le z$, $y\le
z$ (resp.\ $z\le x$, $z\le y$). A~poset is called \emph{directed
complete} (\emph{dcpo} for short) if it has lowest upper bounds for
all its directed subsets.

The~\emph{Scott topology} $\sigma(X)$ on $(X,\le)$ consists of all
those $U\subseteq X$ that satisfy $x\in U~\Leftrightarrow ~U\cap
D\ne\emptyset$ for every $\le-$directed $D\subseteq X$ with a least
upper bound $x.$ Note that ``$\Leftarrow$'' above implies
$U=U\uparrow$.

In a~dcpo $X$, a~set is Scott closed iff it is lower and closed
under suprema of directed subsets.

A mapping $f$ between dcpo's $X$ and $Y$ is
\emph{Scott continuous}, i.e.\ continuous w.r.t.\ $\sigma(X)$ and
$\sigma(Y)$, if and only if it preserves suprema of directed sets.

Let $L$ be a poset. We say that $x$ is \emph{way below} $y$ and
write $x\ll y$ iff, for all directed subsets $D\subseteq L$ such
that $\sup D$ exists, the relation $y\leq \sup D$ implies the
existence of $d\in D$ such that $x\leq d$. ``Way-below'' relation
is transitive and antisymmetric. An~element satisfying $x\ll x$ is
said to be \emph{compact} or \emph{isolated from below}, and in
this case the set $\{x\}\ups$ is Scott open.

A~poset $L$ is called \emph{continuous} if, for each element $y\in
L$, the~set $y\ddns=\{x\in L\mid {x\ll y}\} $ is directed and its
least upper bound is $y$. A~\emph{domain} is a continuous dcpo. If
a~domain is a~semilattice, it is called a~\emph{continuous
semilattice}.

A complete lattice $L$ is called \emph{completely distributive} if,
for each collection of sets $(M_t)_{t\in T}$ in $L$, the equality
$\inf\{\sup M_t\mid t\in T\}=\sup\{\inf\{\alpha_t\mid t\in T\}\mid
(\alpha_t)_{t\in T}\in
\prod_{t\in T}M_t \}$ holds. This property implies distributivity,
but the~converse fails. Then both $L$ and $L^{op}$ are continuous,
and the~join of Scott topologies on $L$ and $L^{op}$ provides
the~unique compact Hausdorff topology on $L$ with~a basis
consisting of small sublattices (\emph{Lawson topology}). In
the~sequel completely distributive lattices will be regarded with
Lawson topologies.

For a~subset $R\subset S_1\times S_2\times \dots\times S_n$,
an~index $k\in\{1,2,\dots,n\}$, and elements $\alpha_1\in S_1$,
\dots, $\alpha_{k-1}\in S_{k-1}$, $\alpha_{k+1}\in S_{k+1}$, \dots, $\alpha_n\in
S_n$, the~set
\begin{gather*}
\underset{\text{\tiny\rm all factors except $\scriptstyle k$-th}}{
\underbrace{
\pr_{1\dots(k-1)(k+1)\dots n}
}}
\bigl(
R\cap (\{\alpha_1\}\times\dots\times\{\alpha_{k-1}\}\times
S_k\times \{\alpha_{k+1}\}\times\dots\times\{\alpha_n\})
\bigr)
\\
=\{\alpha\in S_k\mid (\alpha_1,\dots,\alpha_{k-1},\alpha,\alpha_{k+1}, \dots, \alpha_n)\in R\}
\subset S_k
\end{gather*}
is called a~\emph{cut} of $R$. We denote it
$\alpha_1\dots\alpha_{k-1}R\alpha_{k+1} \dots \alpha_n$.

The~following obvious property is quite useful.

\begin{lemma}\label{lem.scott-cuts}
Let $S_1$, $S_2$, \dots, $S_n$ be {\bfseries dcpo}s, $R\subset
S_1\times S_2\times
\dots\times S_n$. Then $R$ is Scott closed if and only
if all its cuts are Scott closed.
\end{lemma}

\begin{proof}
Necessity is obvious. To prove sufficiency, observe that all cuts
of $R$ being Scott closed (hence lower sets) implies that $R$ is
a~lower set as well. Without loss of generality we can consider
only the~case $n=2$. Let a~subset $D\subset R\subset S_1\times S_2$
be directed. Then the~set $D'=D\dns\subset R$ is directed as well
and lower in $S_1\times S_2$, hence is the~product $D_1\times D_2$
of directed lower sets $D_1\subset S_1$ and $D_2\subset S_2$. For
any $x\in D_1$ the~set $\{x\}\times D_2$ is contained in $R$,
therefore $D_2$ is contained in the~Scott closed cut $xR$. This
implies that the~least upper bound $b=\sup D_2$, which exists
because $S_2$ is a~dcpo, is an~element of this cut, hence $(x,b)\in
R$ for all $x\in D_1$. Thus $D_1$ is contained in the~Scott closed
cut $Rb$, therefore $a=\sup D_1$ in the~dcpo $S_1$ also belongs to
this cut. We obtain that $(a,b)=\sup D'=\sup D$ belongs to $R$ for
any directed subset $D\subset R$, i.e., $R$ is Scott closed.
\end{proof}

Similarly, for a~subset $R\subset S_1\times S_2\times \dots\times
S_n$, an~index $k\in\{1,2,\dots,n\}$, and \emph{subsets}
$A_1\subset S_1$, \dots, $A_{k-1}\subset S_{k-1}$, $A_{k+1}\subset
S_{k+1}$, \dots, $A_n\subset S_n$, we denote
\begin{gather*}
A_1\dots A_{k-1}R A_{k+1} \dots A_n
=\{\alpha\in S_k\mid
\text{there are }
\alpha_1\in S_1, \dots, \alpha_{k-1}\in S_{k-1},
\\
\alpha_{k+1}\in S_{k+1},
\dots, \alpha_n\in S_n
\text{ such that }
(\alpha_1,\dots,\alpha_{k-1},\alpha,\alpha_{k+1}, \dots, \alpha_n)\in R\}.
\end{gather*}

\section{Compatibilities for continuous semilattices}

Let $\Semz$ be the~category of all continuous (meet) semilattices
with zeros and all Scott continuous zero-preserving semilattice
morphisms. We denote by $[S\to S']_0$ the~set of all arrows from
$S$ to $S'$ in $\Semz$.

We use the~results of~\cite{NykMyk:ConMes:10} and denote by
$S^\land$ the~set of all (probably empty) Scott open filters in $S$
except $S$ itself. We order $S^\land$ by inclusion, then $S^\land$
is a~continuous semilattice with the~bottom element $\ES$. Then
$S^\land$ can be regarded as $[S\to\{0,1\}]_0$, i.e., its elements
can be identified with the~bottom-preserving meet-preserving Scott
continuous maps $S\to\{0,1\}$ (the~preimages of $\{1\}$ under such
maps are precisely the~non-trivial Scott open filters in~$S$). For
an~arrow $f:S_1\to S_2$ in $\Semz$ the~formula
$f^\land(F)=f^{-1}(F)$, $F\in S_2^\land$, determines the~mapping
$f^\land:S_2^\land\to S_1^\land$, which is an~arrow in $\Semz$ as
well, hence the~contravariant functor $(-)^{\land}$ in $\Semz$ is
obtained. The~assignment $s\mapsto\{F\in S^\land\mid s\in F\}$ is
an~isomorphism $u_S:S\to S^{\land\land}$ which is a~component of
a~natural transformation $u:\uni{\Semz}\to (-)^{\land\land}$,
hence the~functor $(-)^\land$ is involutive, i.e., is
a~self-duality. In fact it is a~restriction of the~Lawson
duality~\cite{GHK:CLD:03}.

We slightly change the~terminology introduced in
\cite{NykMyk:ConMes:10}:

\begin{definition}
Let $S,S'$ be continuous semilattices with bottom elements
respectively $0,0'$. A~mapping $P:S\times S'\to\{0,1\}$ is called
a~\emph{compatibility} if:

(1) $P$ is distributive w.r.t. $\land$ in the both variables, and
$P(0,y)=P(x,0')=0$ for all $x\in S$, $y\in S'$;

(2) $P$ is Scott continuous.

\medskip

If, additionally, the following holds:

(3) $P$ separates elements of $S$ and of $S'$, i.e.:

(3a) for each $x_1,x_2\in S$, if $P(x_1,y)=P(x_2,y)$ for all $y\in
S'$, then $x_1=x_2$;

(3b) for each $y_1,y_2\in S'$, if $P(x,y_1)=P(x,y_2)$ for all $x\in
S$, then $y_1=y_2$;

then we call $P$ a \emph{separating compatibility}.
\end{definition}

The~definition of (separating) compatibility is symmetric in
the~sense that the~mapping $P':S'\times S\to
\{0,1\}$, $P'(y,x)=P(x,y)$ is a~(separating)
compatibility as well, which we call the~\emph{reverse
compatibility}. For compatibilities we use also infix notation
$xPy\equiv P(x,y)$.

We can consider a~compatibility $P:S\times S'\to\{0,1\}$ as
a~characteristic mapping of a~binary relation $P\subset S \times
S'$, hence it is natural to denote $xP=\{y\in S'\mid xPy=1\}$,
$Py=\{x\in S\mid xPy=1\}$ for all $x\in S$, $y\in S'$.

The~following statement from \cite{NykMyk:ConMes:10} is of crucial
importance:
\begin{proposition}\label{st.polar-SS}
Let $S,S'$ be continuous meet semilattices with bottom elements
$0,0'$ resp. If $P:S\times S'\to \{0,1\}$ is a~separating
compatibility, then the mapping $i$ that takes each $x\in S$ to
$xP$ is an~isomorphism $S\to S'^\land$. Conversely, each
isomorphism $i:S\to S'^\land$ is determined by the~above formula
for a~unique separating compatibility $P:S\times S'\to \{0,1\}$.
\end{proposition}

Similarly, for a~fixed separating compatibility $P:S\times S'\to
\{0,1\}$ and subsets $A\subset S$, $B\subset S'$, the~sets
$$
A^\perp=\{y\in S'\mid xPy=0\text{ for all }x\in A\},
\quad
B^\perp=\{x\in S\mid xPy=0\text{ for all }y\in B\}
$$
will be called the~\emph{transversals} of $A$ and $B$ respectively.

It is easy to see that $A^\perp$ and $B^\perp$ are Scott closed,
and $A^{\perp\perp}=(A^\perp)^\perp$ is the~Scott closure of $A$,
i.e., the~least Scott closed (hence lower) subset in $S$ that
contains $A$, similarly for $B^{\perp\perp}$.

Obviously the~transversal operation $(-)^\perp$ is antitone, i.e.,
$A\subset B$ implies $A^\perp\supset B^\perp$, and for a~filtered
family $\{A_\alpha\mid \alpha\in\CCA\}$ of closed lower sets
the~equality
$$(\bigcap_{\alpha\in\CCA}A_\alpha)^\perp=\Cl(\bigcup_{\alpha\in\CCA}A_\alpha^\perp)$$
is valid. This implies
$$\bigcap_{\alpha\in\CCA}A_\alpha=(\Cl(\bigcup_{\alpha\in\CCA}A_\alpha^\perp))^\perp
=(\bigcup_{\alpha\in\CCA}A_\alpha^\perp)^\perp.
$$

\section{Category of ambiguous representations}

\begin{definition}
Let $S_1,S_2$ be continuous semilattices with zeros.
An~\emph{ambiguous representation} of $S_1$ in $S_2$ is a~binary
relation $R\subset S_1\times S_2$ such that

(a) if $(x,y)\in R$, $x\le x'$ in $S_1$, and $y'\le y$ in $S_2$, then
$(x',y')\in R$ as well;

(b) for all $x\in S_1$ the~set $xR=\{y\in S_2\mid (x,y)\in R\}$ is
non-empty and closed under directed sups in~$S_2$.
\end{definition}

Observe that (a) implies $(x,0_2)\in R$ for all $x\in S$. It also
implies that $xR$ in (b) is a~lower set, hence due to (b) is Scott
closed.

Thus we can rearrange the~requirements as follows, obtaining
an~equivalent definition:

\begin{definition}
For continuous semilattices $S_1,S_2$ with zeros, a~binary relation
$R\subset S_1\times S_2$ is an~\emph{ambiguous representation} of
$S_1$ in $S_2$ if

(a${}'$) for all $x\in S_1$ the~set $xR=\{y\in S_2\mid (x,y)\in
R\}$ is non-empty and Scott closed (i.e., a~directed complete lower
set).

(b${}'$) for all $y\in S_2$ the~set $Ry=\{x\in S_1\mid (x,y)\in
R\}$ is an~upper set.
\end{definition}

We call such ambiguous representations \emph{crisp} to distinguish
them from $L$-fuzzy ambiguous representations that will be defined
in the~next section.

Let separating compatibilities $P_1:S_1\times
\Hat{S_1}\to\{0,1\}$, $P_2=S_2\times
\Hat{S_2}\to\{0,1\}$ be fixed.

For an~ambiguous representation $R\subset S_1\times S_2$ define
the~relation $R^\sms\subset \Hat{S_2}\times \Hat {S_1}$ with
the~formula
$$
R^\sms =
\bigl\{
(\hat y,\hat x)\in \Hat{S_2}\times \Hat{S_1}
\mid
\text{if }xP_1\hat x=1\text{ for some }x\in S_1,
\text{then there is }y\in xR,\; yP_2\hat y=1
\bigr\},
$$
or, equivalently
$$
R^\sms =
\bigl\{
(\hat y,\hat x)\in \Hat{S_2}\times \Hat{S_1}
\mid
\text{if }x\in S_1\text{ is such that }
yP_2\hat y=0\text{ for all }y\in xR,
\text{then } xP_1\hat x=0
\bigr\}.
$$

Obviously $R^\sms$ is an~ambiguous representation as well, hence we
can calculate $R^\ssms=(R^\sms)^\sms\subset S_1\times S_2$ using
the~reverse (in fact the~same) separating compatibilities again.

\begin{proposition}
For each ambiguous representation $R\subset S_1\times S_2$
the~inclusion $R^\ssms\subset R$ holds. The~equality $R=R^\ssms$ is
equivalent to the~condition

{\rm (c) if $(x,y)\in R$ and $y'\ll y$ in $S_2$, then there is
$x'\ll x$ in $S$ such that $(x',y')\in R$. }
\end{proposition}

\begin{proof}
Observe the~validity of the~formula
$$
\hat y R^\sms=
\bigl\{x\in S_1
\bigm|
\hat y\in (xR)^\perp
\bigr\}^\perp.
$$
Thus
$$
\bar x R^\ssms
=
\bigl\{\hat y\in \Hat{S_2}
\bigm|
\bar x\in (\hat y R^\sms)^\perp
\bigr\}^\perp
=
\bigl\{\hat y\in \Hat{S_2}
\bigm|
\bar x\in
\{x\in S
\mid
\hat y\in (xR)^\perp
\}^{\perp\perp}
\bigr\}^\perp.
$$

Taking into account
$$
\bar x\in
\{x\in S
\mid
\hat y\in (xR)^\perp
\}^{\perp\perp}
\Leftarrow
\hat y\in (\bar x R)^\perp,
$$
we obtain
$$
\bar x R^\ssms
\subset
\bigl\{
\hat y\in\Hat{S'}\mid
\hat y\in (\bar x R)^\perp
\bigr\}^\perp=(\bar x R)^{\perp\perp}=\bar x R
$$
for all $\bar x\in S$, which is the~required inclusion.

The~set $\{x\in S\mid\hat y\in (xR)^\perp\}$ is lower, therefore
its double transversal is the~closure, thus
\begin{align*}
\bar x R^\ssms
&=
\bigl\{\hat y\in \Hat{S'}
\bigm|
\bar x\in
\Cl\{x\in S
\mid
\hat y\in (xR)^\perp
\}
\bigr\}^\perp
\\
&=
\bigl\{\hat y\in \Hat{S'}
\bigm|
\hat y\in (xR)^\perp
\text{ for all }x\ll \bar x
\bigr\}^\perp
\\
&=
\bigl(
\bigcap\{(xR)^\perp\mid x\ll \bar x\}
\bigr)^\perp.
\end{align*}

The~family $\{\{xR)^\perp\mid x\ll \bar x\}$ of closed lower sets
is filtered, hence the~transversal of its intersection equals
$$
\Cl
\bigl(
\bigcup\{(xR)^{\perp\perp}\mid x\ll \bar x\}
\bigr)
=
\Cl
\bigl(
\bigcup\{xR\mid x\ll \bar x\}
\bigr),
$$
and the~equality $R^\ssms=R$ is equivalent to
$$
\bar xR
=
\Cl
\bigl(
\bigcup\{xR\mid x\ll \bar x\}
\bigr),
$$
which in fact is the~condition~(c).
\end{proof}

The~equality $R\ssms=R$ implies $(R^\sms)^\ssms=R^\sms$, hence, if
(c) holds for $R$. then it holds for $R^\sms$ as well. Therefore on
such ambiguous representations the~operation $(\;\;)^\sms$ is
involutive, and we call each of $R$ and $R^\sms$
the~\emph{pseudo-inverse} to the~other one. The ambiguous
representations satisfying (c) are called \emph{pseudo-invertible}.

One of the~reasons to consider this subclass is that, if we compose
ambiguous representations as relations, i.e., for $R\subset
S_1\times S_2$, $Q\subset S_2\times S_3$:
$$
RQ=\{(x,z)\in S_1\times S_3\mid \text{there is }y\in S_2
\text{ such that }(x,y)\in R, (y,z)\in Q\},
$$
then the~resulting relation can fail to satisfy closedness in
the~condition (b) of the~definition of ambiguous representation,
hence ambiguous representations do not form a~category.

To improve things, redefine the~composition as
\begin{gather*}
R{;}Q=\{(x,z)\in S_1\times S_3\mid z\in\Cl(xRQ)\}
\\
=
\{(x,z)\in S_1\times S_3\mid
\text{for all }z'\ll z
\text{ there is }y\in S_2
\text{ such that }(x,y)\in R, (y,z')\in Q\}.
\end{gather*}
Now closedness is at hand, but the composition ``$;$'' is not
associative.

\begin{lemma}
For all ambiguous representations $R\subset S_1\times S_2$,
$Q\subset S_2\times S_3$ the~inclusion $Q^\sms{;}R^\sms\subset
(R{;}Q)^\sms$ holds.
\end{lemma}

\begin{proof}
Since $\hat z Q^\sms{;}R^\sms=\Cl(zQ^\sms R^\sms)$ for all $\hat
z\in\Hat{S_3}$, and $\hat z(R{;}Q)^\sms$ is closed in $S_1$, it is
sufficient to prove $Q^\sms R^\sms\subset(R{;}Q)^\sms$.

If $(\hat z,\hat x)\in Q^\sms R^\sms$, then choose $\hat y\in
\Hat{S_2}$ such that $(\hat z,\hat y)\in Q^\sms$ and
$(\hat y,\hat x)\in R^\sms$, and combine
\begin{gather*}
\text{if }xP_1\hat x=1\text{ for some }x\in S_1,
\text{then there is }y\in xR,\; yP_2\hat y=1,
\\
\text{if }yP_2\hat y=1\text{ for some }y\in S_2,
\text{then there is }z\in yQ,\; zP_3\hat z=1
\end{gather*}
to obtain
$$
\text{if }xP_1\hat x=1\text{ for some }x\in S_1,
\text{then there is }z\in xRQ,\; zP_3\hat z=1.
$$
Moreover the~latter $z$ belongs to $xR{;}Q$, hence $(\hat z,\hat
x)\in (R{;}Q)^\sms$.
\end{proof}

\begin{corollary}
Let ambiguous representations $R\subset S_1\times S_2$, $Q\subset
S_2\times S_3$ be pseudo-invertible. Then $Q^\sms{;}R^\sms=
(R{;}Q)^\sms$, and the~composition $R{;}Q$ is pseudo-invertible as
well.
\end{corollary}

\begin{proof}
By the~above and taking into account that $(-)^\sms$ is isotone:
$$
R{;}Q=R^\ssms{;}Q^\ssms\subset (Q^\sms{;}R^\sms)^\sms
\subset
(R{;}Q)^\ssms,
$$
and the~reverse inclusion is known, hence
$R{;}Q=(R{;}Q)^\ssms=(Q^\sms{;}R^\sms)^\sms$, i.e., $R{;}Q$ is
pseudo-invertible. Apply $(-)^\sms$ to this again and obtain
$(R{;}Q)^\sms=(Q^\sms{;}R^\sms)^\ssms=Q^\sms{;}R^\sms$.
\end{proof}

\begin{proposition}
Composition ``;'' of the~pseudo-invertible ambiguous
representations is associative.
\end{proposition}

\begin{proof}
Recall that, for ambiguous representations $R\subset S_1\times
S_2$, $Q\subset S_2\times S_3$, the~composition is calculated as
$$
R{;}Q=\{(x,z)\in S_1\times S_3\mid
\text{for all }z'\ll z
\text{ there is }y\in S_2
\text{ such that }(x,y)\in R, (y,z')\in Q
\}.
$$
It is also important that, for elements $a\ll c$ in a continuous
semilattice, there is an~element $b$ such that~$a\ll b\ll c$.

Now, let $(x,t)\in R{;}(Q{;}T)$. For any $t'\ll t$ choose $t''$
such that $t'\ll t''\ll t$, then there is $y\in S_2$ such that
$(x,y)\in R$, $(y,t'')\in Q{;}T$. The~latter implies that there is
$z\in S_3$ such that $(y,z)\in Q$, $(z,t')\in T$.

Similarly, let $(x,t)\in (R{;}Q){;}T$, then for all $t'\ll t$
choose $t'\ll t''\ll t$, and there is $z'\in S_3$ such that
$(x,z')\in R{;}Q$, $(z',t'')\in T$. Pseudo-invertibility of $T$
implies the~existence of $z\ll z'$ such that $(z,t')\in T$ as well.
There is also $y\in S_2$, $(x,y)\in R$, $(y,z)\in Q$.

On the~other hand, if, for $x\in S_1$, $t\in S_4$, elements $y\in
S_2$, $z\in S_3$ exist for all $t'\ll t$ such that $(x,y)\in R$,
$(y,z)\in Q$, $(z,t')\in T$, then $(x,t')\in RQT$, hence $(x,t')\in
R(Q{;}T)$ and $(x,t')\in (R{;}Q)T$, which in turn implies both
$(x,t)\in R{;}(Q{;}T)$ and $(x,t)\in (R{;}Q){;}T$. Thus
$R{;}(Q{;}T)=(R{;}Q){;}T$.
\end{proof}

It is easy to verify that, for a~continuous semilattice $S$ with
zero, the~relation $E_S=\{(x,y)\in S\times S\mid y\le x\}$ is
a~pseudo-invertible ambiguous representation that is a~neutral
element for composition. Thus:

\begin{proposition}
All continuous semilattices with bottom elements and all
pseudo-invertible ambiguous representation form a~category
$\Sempr$, which contain $\Semz$ as a~subcategory.
\end{proposition}

An~obvious embedding $\Semz\to\Sempr$ is of the~form: $IS=S$ for
an~object $S$, and $If=\{(x,y)\in S_1\times S_2\mid y\le f(x)\}$
for an~arrow $f:S_1\to S_2$.

We denote an~arrow $R$ from $S_1$ to $S_2$ in $\Sempr$ with
$R:S_1\Rightarrow S_2$ and use for the~composition of $R$ and $Q$
the~synonymic notations $R{;}Q$ (in direct order) and $Q\circ R$
(in reverse order) both in $\Semz$ and $\Sempr$.

\begin{proposition}
The~correspondence $(-)^\sms$ is an~involutive contravariant
functor (a~self-duality) in $\Sempr$, which is an~extension of
the~functor $(-)^\land$ in $\Semz$.
\end{proposition}

\section{Category of $L$-fuzzy ambiguous representations}

Now we extend the notion of ambiguous representation to
lattice-valued relations in the spirit of \cite{Nyk:CapLat:08} and
\cite{NykMyk:ConMes:10}.

\begin{definition}
Let $S_1,S_2$ be continuous semilattices with zeros $0_1$ and $0_2$
resp., $L$ a~completely distributive lattice with a~bottom element
$0$ and a~top element $1$. An~\emph{$L$-fuzzy ambiguous
representation} (or an~$L$-\emph{ambiguous representation} for
short) of $S_1$ in $S_2$ is a~ternary relation $R\subset S_1\times
S_2\times L$ such that

(a) if $(x,y,\alpha)\in R$, $x\le x'$ in $S_1$, $y'\le y$ in $S_2$,
and $\alpha'\le \alpha$ in $L$, then $(x',y',\alpha')\in R$ as
well;

(b) for all $x\in S_1$ the~set $xR=\{(y,\alpha)\in S_2\times L\mid
(x,y,\alpha)\in R\}$ is Scott closed in~$S_2\times L$ and contains
all the~elements of the forms $(0_2,\alpha )$ and $(y,0)$;

(c) for all $x\in S_1$, $y\in S_2$, $\alpha,\beta\in L$, if
$(x,y,\alpha)\in R$, $(x,y,\beta)\in R$, then
$(x,y,\alpha\lor\beta)\in R$.
\end{definition}

The following definition is equivalent.

\begin{definition}
For continuous semilattices $S_1,S_2$ with zeros and a~completely
distributive lattice $L$, a~ternary relation $R\subset S_1\times
S_2\times L$ is an~$L$-\emph{ambiguous representation} of $S_1$ in
$S_2$ if

(a${}'$) for all $y\in S_2$, $\alpha\in L$ the~set $Ry\alpha=\{x\in
S_1\mid (x,y,\alpha)\in R\}$ is an~upper set in~$S_1$;

(b${}'$) for all $x\in S_1$, $\alpha\in L$ the~set $xR\alpha=\{y\in
S_2\mid (x,y,\alpha)\in R\}$ is non-empty and Scott closed in
$S_2$;

(c${}'$) for all $x\in S_1$, $y\in S_2$ the~set $xyR=\{\alpha\in
L\mid (x,y,\alpha)\in R\}$ is non-empty, directed, and Scott closed
in $L$.
\end{definition}

Obviously, due to complete distributivity of $L$, (c${}'$) is
equivalent to any of the~following properties:

(c${}''$) for all $x\in S_1$, $y\in S_2$ the~set $xyR=\{\alpha\in
L\mid (x,y,\alpha)\in R\}$ is a~non-empty directed lower set such
that, if $\beta\in xyR$ for all $\beta\ll\alpha$, then $\alpha\in
xyR$; or

(c${}'''$) for all $x\in S_1$, $y\in S_2$ the~set $xyR$ is a~lower
set with a~greatest element (i.e., a~set of the form
$\{\alpha\}\dns$).

Observe also that (a${}'$)+(b${}'$) mean that, for all $\alpha\in
L$, the~cut $R\alpha=\{(x,y)\in S_1\times S_2\mid (x,y,\alpha)\in
R\}$ is a~(crisp) ambiguous representation of $S_1$ in $S_2$ as
defined in the~previous section. We will call it
the~$\alpha$-\emph{cut} of $R$ and denote $R_\alpha$.

We assume again that separating compatibilities $P_1:S_1\times
\Hat{S_1}\to\{0,1\}$, $P_2=S_2\times
\Hat{S_2}\to\{0,1\}$ are fixed.

For an~ambiguous representation $R\subset S_1\times S_2\times L$
define the~relation $R^\sms\subset \Hat{S_2}\times \Hat {S_1}\times
L$ through its $\alpha$-cuts as follows:
$$
(R^\sms)_\alpha=\bigcap_{\beta\ll\alpha}(R_\beta)^\sms,
$$
or, equivalently, with the~formulae
\begin{multline*}
R^\sms =
\bigl\{
(\hat y,\hat x,\alpha)\in \Hat{S_2}\times \Hat{S_1}
\mid
\text{if }\beta\ll\alpha
\text{ and }xP_1\hat x=1\text{ for some }x\in S_1,
\\
\text{then there is }(y,\beta)\in xR,\; yP_2\hat y=1
\bigr\},
\end{multline*}
or
\begin{multline*}
R^\sms =
\bigl\{
(\hat y,\hat x)\in \Hat{S_2}\times \Hat{S_1}
\mid
\text{if }\beta\ll\alpha
\text{ and }x\in S_1\text{ is such that }
yP_2\hat y=0
\\
\text{for all }(y,\beta)\in xR,
\text{then } xP_1\hat x=0
\bigr\}.
\end{multline*}

A~shorter formula uses transversals:
$$
\hat y (R^\sms)_\alpha=
\bigcap_{\beta\ll\alpha}
\bigl\{x\in S_1
\bigm|
\hat y\in (xR_\beta)^\perp
\bigr\}^\perp.
$$

\begin{proposition}
The~relation $R^\sms$ is an~$L$-ambiguous representation as well.
\end{proposition}

\begin{proof}
For the~intersection of crisp ambiguous representations is a~crisp
ambiguous representation as well, (a${}'$)+(b${}'$) for $R^\sms$
are immediate. To verify c${}''$, assume that $\beta\in
\hat y\hat xR^\sms$ for all $\beta\ll\alpha$, i.e.,
$(\hat y,\hat x)\in\bigcap_{\beta\ll\alpha}
\bigcap_{\gamma\ll\beta}(R_\gamma)^\sms$. In a~completely
distributive lattice $L$ we have $\gamma\ll\alpha$ if and only if
there is $\beta$ such that $\gamma\ll\beta\ll\alpha$, hence $(\hat
y,\hat x)\in\bigcap_{\gamma\ll\alpha}(R_\gamma)^\sms$, which
implies $\alpha\in \hat y\hat x R^\sms$.

Obviously $\hat y\hat x R^\sms$ is a~lower set that contains $0$.
Show that it is directed. If $\alpha,\beta\in \hat y\hat xR^\sms$,
then for all $\gamma\ll\alpha\lor\beta$ there are
$\alpha'\ll\alpha$, $\beta'\ll\beta$ such that
$\gamma\le\alpha'\lor\beta'$. Therefore
\begin{align*}
\hat x&\in
\hat y(R_{\alpha'})^\sms\cap \hat y(R_{\beta'})^\sms
\\
&=
\bigl\{x\in S_1
\bigm|
\hat y\in (xR_{\alpha'})^\perp
\bigr\}^\perp
\cap
\bigl\{x\in S_1
\bigm|
\hat y\in (xR_{\beta'})^\perp
\bigr\}^\perp
\\
&=
\bigl\{x\in S_1
\bigm|
\hat y\in (xR_{\alpha'})^\perp
\text{ or }
\hat y\in (xR_{\beta'})^\perp
\bigr\}^\perp
\\
&\subset
\bigl\{x\in S_1
\bigm|
\hat y\in (xR_{\alpha'\lor\beta'})^\perp
\bigr\}^\perp
=\hat y (R_{\alpha'\lor \beta'})^\sms
\subset \hat y (R_{\gamma})^\sms.
\end{align*}

Thus $\alpha\lor\beta\in \hat y\hat x R^\sms$, and the~latter set
is directed in~$L$.
\end{proof}

Hence $R^\ssms=(R^\sms)^\sms\subset S_1\times S_2\times L$ is
an~$L$-ambiguous representation as well.

\begin{proposition}
For each ambiguous representation $R\subset S_1\times S_2$
the~inclusion $R^\ssms\subset R$ holds. The~equality $R=R^\ssms$ is
equivalent to the~condition

{\rm (d)} if $(x,y,\alpha)\in R$, $y'\ll y$ in $S_2$, and
$\alpha'\ll\alpha$ in $L$, then there is $x'\ll x$ in $S$ such that
$(x',y',\alpha')\in R$.
\end{proposition}

\begin{proof}
Recall that
$$
\hat y (R^\sms)_\beta
=
\bigcap_{\gamma\ll\beta}
\bigl\{x\in S_1
\bigm|
\hat y\in (xR_\gamma)^\perp
\bigr\}^\perp
=
\bigl(
\bigcup_{\gamma\ll\beta}
\bigl\{x\in S_1
\bigm|
\hat y\in (xR_\gamma)^\perp
\bigr\}
\bigr)^\perp.
$$

Hence
\begin{align*}
x (R^\ssms)_\alpha
&=
\bigcap_{\beta\ll\alpha}
\bigl\{\hat y\in \hat S_2
\bigm|
x\in (\hat y (R^\sms)_\beta)^\perp
\bigr\}^\perp
\\
&=
\bigcap_{\beta\ll\alpha}
\Bigl\{\hat y\in \hat S_2
\bigm|
x\in
\bigl(
\bigcup_{\gamma\ll\beta}
\bigl\{x'\in S_1
\bigm|
\hat y\in (x'R_\gamma)^\perp
\bigr\}
\bigr)^{\perp\perp}
\Bigr\}^\perp
\\
&\text{\ \ \ \ \ (double transversal is Scott closure)}
\\
&=
\bigcap_{\beta\ll\alpha}
\Bigl\{\hat y\in \hat S_2
\bigm|
x\in
\Cl\bigl(
\bigcup_{\gamma\ll\beta}
\bigl\{x'\in S_1
\bigm|
\hat y\in (x'R_\gamma)^\perp
\bigr\}
\bigr)
\Bigr\}^\perp
\\
&\text{\mbox{}\hskip-2em(Scott closure of a lower set $A$ consists of all points approximated by elements of $A$)}
\\
&=
\bigcap_{\beta\ll\alpha}
\Bigl\{\hat y\in \hat S_2
\bigm|
\text{for all }x'\ll x
\text{ there is }\gamma\ll\beta
\text{ such that }
\hat y\in (x' R_\gamma)^\perp
\Bigr\}^\perp
\\
&\mathrel{\overset*=}
\bigcap_{\beta\ll\alpha}
\Bigl\{\hat y\in \hat S_2
\bigm|
\text{for all }x'\ll x
\quad
\hat y\in (x' R_\beta)^\perp
\Bigr\}^\perp
\\
&=
\bigcap_{\beta\ll\alpha}
\Bigl(
\{x\}\ddns R_\beta\Bigr)^{\perp\perp}
\subset
\bigcap_{\beta\ll\alpha}
\bigl(
x R_\beta\bigr)^{\perp\perp}
=
\bigcap_{\beta\ll\alpha}
\bigl(
x R_\beta\bigr)
=
xR_\alpha,
\end{align*}
therefore $R^\ssms\subset R$. Clarify why the~``$=$''
sign with an~asterisk is valid. Obviously, if $\gamma\ll\beta$ and $\hat y\in (x'
R_\gamma)^\perp$, then $\hat y\in (x' R_\beta)^\perp$, and
$({-})^\perp$ is antitone, hence ``$\supset$'' is immediate. On
the~other hand,
\begin{multline*}
\text{there is }\gamma\ll\beta
\text{ such that for all }x'\ll x
\text{ \ }
\hat y\in (x' R_\gamma)^\perp
\\
\implies
\text{for all }x'\ll x
\text{ there is }\gamma\ll\beta
\text{ such that }
\hat y\in (x' R_\gamma)^\perp,
\end{multline*}
hence
\begin{gather*}
\bigcap_{\beta\ll\alpha}
\Bigl\{\hat y\in \hat S_2
\bigm|
\text{for all }x'\ll x
\text{ there is }\gamma\ll\beta
\text{ such that }
\hat y\in (x' R_\gamma)^\perp
\Bigr\}^\perp
\\
\subset
\bigcap_{\beta\ll\alpha}
\Bigl\{\hat y\in \hat S_2
\bigm|
\text{there is }\gamma\ll\beta
\text{ such that for all }x'\ll x
\text{ \ }
\hat y\in (x' R_\gamma)^\perp
\Bigr\}^\perp
\\
=
\bigcap_{\gamma\ll\beta\ll\alpha}
\Bigl\{\hat y\in \hat S_2
\bigm|
\text{for all }x'\ll x
\text{ \ }
\hat y\in (x' R_\gamma)^\perp
\Bigr\}^\perp
\\
=
\bigcap_{\gamma\ll\alpha}
\Bigl\{\hat y\in \hat S_2
\bigm|
\text{for all }x'\ll x
\text{ \ }
\hat y\in (x' R_\gamma)^\perp
\Bigr\}^\perp,
\end{gather*}
therefore ``$\subset$'' (after renaming $\gamma$ with $\beta$) is
also obtained.

It is also clear that, for $R^\ssms=R$ to be valid, it is
necessary and sufficient that the~only ``$\subset$'' in the~above
sequence is ``$=$'', i.e., each $y\in xR_\alpha$ must be a~closure
point of all lower sets $\{x\}\ddns R_\beta$ for
$\beta\ll\alpha$, which in fact is (d).
\end{proof}

The~$L$-ambiguous representations $R$ that satisfy $R=R^\ssms$ are
also called \emph{pseudo-invertible}.

To define composition of $L$-ambiguous representations, we need
an~additional operation ${*}:L\times L\to L$ that makes $L=(L,*)$
a~\emph{unital quantale}, i.e., this operation is infinitely
distributive w.r.t.\ supremum in both variables (hence Scott
continuous) and $1$ is a~two sided unit for~``${*}$''. Note that
\emph{commutativity} is not demanded, hence from now on ``${*}$'' is
a~(possibly) noncommutative lower semicontinuous conjunction for
an~$L$-valued fuzzy logic. The~operation $\alpha\hast\beta\equiv
\beta*\alpha$ satisfies the~same requirements, hence $\hat L=(L,{\hast})$
is a~unital quantale as well.

Then, for $L$-ambiguous representations $R\subset S_1\times
S_2\times L$ and $Q\subset S_2\times S_3\times L$, the~composition
$R*Q$ can be defined in a~manner usual for $L$-fuzzy relations:
\begin{multline*}
R*Q=\bigl\{(x,z,\alpha)\in S_1\times S_3\times L
\mid
\\
\alpha\le\sup\{\beta*\gamma\mid
\text{there is }y\in S_3
\text{ such that }(x,y,\beta)\in R,(y,z,\gamma)\in Q\}
\bigr\}.
\end{multline*}

Similarly to the~case of crisp ambiguous representations, this
composition is associative, but $R*Q$ can fail to be
an~$L$-ambiguous representation, namely (b) is not always valid.
Therefore we ``improve'' the~composition as follows: $R{\astsc}
Q\subset S_1\times S_3\times L$ is such that $xR{\astsc}
Q=\Cl(xR*Q)$ for all $x\in S_1$. Here is an~expanded version of
the~latter definition: for $x\in S_1$, $z\in S_3$, $\alpha\in L$ we
have $(x,z,\alpha)\in R{\astsc} Q$ if and only if for all $z'\ll
z$, $\alpha'\ll \alpha$ there are $n\in\BBN$, $y_1,\dots,y_n\in
S_2$, $\beta_1,\gamma_1,\dots,\beta_n,\gamma_n\in L$ such that
\begin{gather*}
(x,y_1,\beta_1),\dots,(x,y_n,\beta_n)\in R,\;
(y_1,z',\gamma_1),\dots,(y_n,z',\gamma_n)\in Q,
\\
\beta_1*\gamma_1\lor\dots\lor\beta_n*\gamma_n\ge\alpha'.
\end{gather*}

The~simplest operation ``$*$'' that obviously satisfies the~above
conditions is the~lattice meet ``$\land$''. In this case for
``$\astsc$'' we use the~denotation~``$;$''.

For an~$L$-ambiguous representation $R$ we regard $R^\sms$ as
a~$\hat L$-ambiguous representation, and use ``$\hastsc$'' for
the~compositions of such representations.

\begin{lemma}
For all $L$-ambiguous representations $R\subset S_1\times S_2\times
L$, $Q\subset S_2\times S_3\times L$ the~inclusion
$Q^\sms{\hastsc}R^\sms\subset (R{\astsc}Q)^\sms$ holds.
\end{lemma}

\begin{proof}
Let $(\hat z,\hat x,\alpha)\in Q^\sms{\hastsc}R^\sms$. For each
$\alpha'\ll\alpha$ we can choose $\alpha''$ such that
$\alpha'\ll\alpha''\ll\alpha$, then for all $\hat x'\ll
\hat x$ there are $n\in\BBN$, $\hat y_1,\dots,\hat y_n\in S_2$,
$\gamma_1,\beta_1,\dots,\gamma_n,\beta_n\in L$ such that
\begin{gather*}
(\hat z,\hat y_1,\gamma_1),\dots,(\hat z,\hat y_n,\gamma_n)\in Q^\sms,\;
(\hat y_1,\hat x',\beta_1),\dots,(\hat y_n,\hat x',\beta_n)\in R^\sms,
\\
\gamma_1\hast\beta_1\lor\dots\lor\gamma_n\hast\beta_n=
\beta_1*\gamma_1\lor\dots\lor\beta_n*\gamma_n\ge\alpha''.
\end{gather*}

For all $x\in S_1$ such that $xP_1\hat x'=1$, and all
$\beta_1'\ll\beta_1$, \dots, $\beta_n'\ll\beta_n$ there are
$y_1,\dots,y_n\in S_2$ such that
$(x',y_1,\beta'_1),\dots,(x',y_n,\beta'_n)\in R$, $y_1P_2\hat
y_1=\dots=y_nP_2\hat y_n=1$. Analogously, for all
$\gamma_1'\ll\gamma_1$, \dots, $\gamma_n'\ll\gamma_n$ there are
$z_1,\dots,z_n\in S_3$ such that
$(y_1,z_1,\gamma'_1),\dots,(y_n,z_n,\gamma_n')\in Q$, $z_1P_3\hat
z=\dots=z_nP_3\hat z=1$. Then the~element $z=z_1\land\dots\land
z_n$ satisfies $zP_3\hat z=1$ and
$(y_1,z,\gamma'_1),\dots,(y_n,z,\gamma_n')\in Q$ as well.

Obviously
$(x,z,\beta_1'*\gamma_1'\lor\dots\lor\beta_n'*\gamma_n')\in R*Q$.
Due to Scott continuity (i.e., lower semicontinuity) of $*$ and
$\lor$, we can choose the~above
$\beta'_1,\dots,\beta_n',\gamma_1',\dots,\gamma_n'$ so that
$\beta_1'*\gamma_1'\lor\dots\lor\beta_n'*\gamma_n'\ge\alpha'$.
Hence for all $\alpha'\ll\alpha$, $\hat x'\ll\hat x$, $x\in S_1$,
$xP_1\hat x'$ there is $z\in S_3$ such that $zP_3\hat z=1$,
$(x,z,\alpha')\in R{\astsc}Q$, i.e., $(\hat z,\hat x,\alpha)\in
(R{\astsc}Q)^\sms$.

\end{proof}

\emph{Mutatis mutandis} we obtain an~analogue of a~statement for
crisp ambiguous representations:

\begin{corollary}
Let $L$-ambiguous representations $R\subset S_1\times S_2\times L$,
$Q\subset S_2\times S_3\times L$ be pseudo-invertible. Then
$Q^\sms{\hastsc}R^\sms= (R{\astsc}Q)^\sms$, and the~composition
$R{\astsc}Q$ is pseudo-invertible as well.
\end{corollary}

\begin{proposition}
Composition ``${\astsc}$'' of the~pseudo-invertible $L$-ambiguous
representations is associative.
\end{proposition}

\begin{proof} is similar to the~one for crisp
representations and reduces to the~observation that, for
$L$-ambiguous representations $R\subset S_1\times S_2\times L$,
$Q\subset S_2\times S_3\times L$, $S\subset S_3\times S_4\times L$,
both statements $(x,t,\alpha)\in (R{\astsc}Q){\astsc}S$ and
$(x,t,\alpha)\in R{\astsc}(Q{\astsc}S)$ are equivalent to
the~existence, for all $\alpha'\ll\alpha$ and $t'\ll t$, of
$n\in\BBN$, $y_1,\dots,y_n\in S_2$, $z_1,\dots,z_n\in S_3$,
$\beta_1,\gamma_1,\delta_1,\dots,\beta_n,\gamma_n,\delta_n\in L$
such that
\begin{gather*}
(x,y_1,\beta_1),\dots,(x,y_n,\beta_n)\in R,
\quad
(y_1,z_1,\gamma_1),\dots,(y_n,z_n,\gamma_n)\in Q,
\\
(z_1,t',\delta_1),\dots,(z_n,t',\delta_n)\in S,
\quad
\beta_1*\gamma_1*\delta_1\lor\dots\lor\beta_n*\gamma_n*\delta_n\ge\alpha'.
\end{gather*}
\end{proof}

Hence we obtain a~category:

\begin{proposition}
All continuous semilattices with bottom elements and all
pseudo-invertible $L$-ambiguous representations form a~category
$\Sempral$, which contains $\Sempr$ as a~subcategory.
\end{proposition}

The~embedding $I^*_L:\Sempr\to\Sempral$ preserves the~objects and
turns each crisp ambiguous representation $R\subset S_1\times S_2$
into an~$L$-ambiguous one as follows:
$$
I^*_L R=\{(x,y,\alpha)\subset S_1\times S_2\times L\mid (x,y)\in R
\text{ or }\alpha=0\}.
$$

Clearly there is also the~embedding $I^{\hat
*}_L:\Sempr\to\Semprhl$ into the~category built upon
the~``swapped'' operation ``${\hat *}$''.

We denote an~arrow $R$ from $S_1$ to $S_2$ in $\Sempral$ with
$R:S_1\Rightarrow^* S_2$. The~composition of $R$ and $Q$ is denoted
by $R{\astsc}Q$ (in direct order) or $Q\circledast R$ (in reverse
order) in $\Sempral$, and by $R{\hastsc}Q$ or $Q\hoast R$
respectively in~$\Semprhl$.

\begin{proposition}
The~self-duality $(-)^\sms:\Sempr\to\Sempr$ extends to
contravariant functors $(-)^\sms:\Sempral\to\Semprhl$ and
$(-)^\sms:\Semprhl\to\Sempral$. Both pairwise compositions of these
functors are isomorphic to the~identity functors.
\end{proposition}

Hence $(-)^\sms$ is a~``contravariant equivalence'' between
$\Sempral$ and $\Semprhl$.

\section{Examples and interpretation}

Although the above arguments are formally correct, motivation for
introduction of so complicated objects and constructions is rather
obscure. Here we present one of the~possible interpretations to
advocate the proposed theory.

Consider a~system which can be in different states, and these
states can change, e.g., a~game, position in which changes after
moves of players. Assume that any party involved (e.g., a~player)
at a moment of time can obtain only imperfect information about
the~state of the~system (``imperfect'' means that this information
only reduces uncertainty but not necessarily eliminates it). To
each such observation we assign a~continuous meet semilattice $S$
with zero, and its~elements are regarded as possible pieces of
information (or statements) about the~state of the~system (cf.
\cite{Scott:DDS:82} for a~detailed explanation why continuous
semilattices are an~appropriate tool for this purpose, and
\cite{GHK:CLD:03,Jstn:StSp:83} for more information on continuous
posets). If $x\le y$ in $S$, then the~information $x$ is more
specific (restrictive) than $y$. Respectively $0\in S$ means ``no
information''. The~meet of $x$ and $y$ is the~most specific piece
of information including both $x$ and $y$ (as particular cases). It
is not necessarily equivalent to ``$x$ or $y$'' in the~usual
logical sense.
\begin{example}
Let $S$ be the~set of all non-empty segments $[a,b]\subset [0,1]$
ordered by reverse inclusion. We regard $[a,b]$ as the~statement
``only points of $[a,b]$ have been selected''. Then
$[0,\frac13]\land [\frac23,1]=[0,1]$, which is wider than
``$[0,\frac13]$ or $[\frac23,1]$''.
\end{example}

The~simplest non-trivial such semilattice is $\{0,1\}$, which has
many natural interpretations, e.,g., $0=$``maybe possible'',
$1=$``surely impossible'' or $0=$``game outcome is unknown'',
$1=$``player wins''.

Another important case is when $S$ is the~hyperspace $\exp X$ of
all non-empty closed sets of a~compactum~$X$. Points of $X_i$ are
possible states of the~system, and $A\subcl X_i$ represents the
fact that one of the~states $x\in A$ is achieved. Then $\exp X$ is
ordered by reverse inclusion, hence $X\in\exp X$ is the~least
element that means ``anything can happen''. The~meet of $A$ and $B$
is their union, therefore can be interpreted as ``$A$ or $B$''.
The~ambiguous representations considered in the~two previous
sections were first introduced and their properties proved for this
particular case in~\cite{NykRep:AmbRep:11}.

A~crisp ambiguous representation $R\subset S\times S'$, i.e.,
an~arrow $R:S\Rightarrow S'$ in $\Sempr$, consists of all pairs
$(x,x')$ such that, given an~information $x$ in the~observation
$S$, one can assure that $x'$ will be valid in the~observation
$S'$. It is easy to understand why $y\ge x$, $y'\le x'$, and $xRx'$
imply $y'Ry'$, and why $xR0$ always holds. If we can observe $x'$
but not $x$, then the~``visible'' information $x'$ represents
the~``hidden'' information~$x$, which is not necessarily unique,
therefore the~term ``ambiguous representation'' has been chosen. In
particular, an~ambiguous representation models a~move in the~game
with imperfect information~: one imperfect knowledge on
the~complete state is replaced with another one.

In particular, an~ambiguous representation $R:S\Rightarrow\{0,1\}$
may mean that all statements $x\in S$ such that $xR1$ are
impossible, and the~remaining ones are probably (but not
necessarily) possible, or that any piece $x$ of information such
that $xR1$ guarantees that the~player can win (depending on
the~chosen interpretation of $\{0,1\}$).

Recall that $\Semz$ is a~subcategory of $\Sempr$, with
the~morphisms being meet-preserving zero-preserving Scott
continuous mappings. An~ambiguous representation $R:S\Rightarrow
S'$ in $\Sempr$ belongs to $\Semz$ if and only if $x_1Ry_1$,
$x_2Ry_2$, $y\le y_1$, and $y\le y_2$ imply existence of $x$ such
that $x\le x_1$, $x\le x_2$, and $xRy$. In other words, $R$ is
consistent with the~interpretation of ``$\land$'' as ``or'' both in
$S$ and $S'$~: if, given $x_i$, one may obtain $y_i$, $i=1,2$, then
``$y_1$ or $y_2$'' can be obtained from ``$x_1$ or $x_2$''.

In particular, an arrow $R:S\to\{0,1\}$ in $\Semz$ declares some
elements of $S_i$ impossible in such a~way that if $x_1$ and $x_2$
are impossible, then there is a~statement $x$ that includes $x_1$
and $x_2$ and is impossible as well.

\begin{example}
Let $S$ be the~previously defined semilattice of all segments
$[a,b]\subset[0,1]$, $R,R'\subset S\times\{0,1\}$ be the~ambiguous
representations defined as follows:
$$
[a,b]R1\iff [a,b]\not\ni\frac12,\qquad[a,b]R'1\iff
[a,b]\subset(\frac13,\frac23),
$$
i.e., $R$ requires that the~point $\frac12$ must be selected, and
$R'$ demands that among the~selected points must be a~point
$\le\frac13$ or $\ge\frac23$. It is easy to see that $R'$ but not
$R$ is an~arrow in $\Semz$.
\end{example}

Cartesian product $S\times S'\times\dots$ represents ``joint
knowledge'' about states of the~system observed at
different~moments or by different parties. An arrow $P:S\times
S'\Rightarrow\{0,1\}$ carries an~information on simultaneous
realization of statements $x\in S$, $y\in S'$~: if $(x,y)P1$, then
it is impossible. If we additionally require the~mapping
$(x,y)\mapsto(x,y)P1$, which we denote with the~same letter $P$, to
be meet-preserving in each argument, then the~definition of
compatibility is obtained. Meet-preservation means that if $x_1$
and $x_2$ are incompatible (cannot be valid together) with $y$,
then ``$x_1$ or $x_2$'' is incompatible with $y$ as well, similarly
for the~second argument. A compatibility $P:S\times S'\to \{0,1\}$
is separating if for all $x_1\ne x_2$ in $S$ there is $y\in S'$
such that exactly one of $x_i$ is incompatible with $y$ w.r.t.\
$P$, similarly for $y_1\ne y_2$ in $S'$ and $x\in S$. Then elements
of $S'$ can be regarded as ``negative statements'' about state of
the~game observed at $S$~: given $y\in S'$ and a~separating
compatibility $P$, we declare impossible all $x\in S$ such that
$(x,y)P1$. For each $A\subset S'$ its transversal $A^\perp=\{x\in
S\mid xPy=0\text{ for all }y\in A\}$ consists of all statements in
$S$ that are compatible with all ``negative statements'' from~$A$.
Hence elements of $S'$ prevent or prohibit elements of $S$, and
vice versa.

\begin{example}
Let $L$ be a~completely distributive lattice, then so is $\tilde
L=L^{op}$. The~mapping $P:L\times L^{op}\to\{0,1\}$ defined as
$$
P(x,y)=
\begin{cases}
0, & y\not\ll x,
\\
1, & y\ll x,
\end{cases}
$$
is a~separating compatibility. In particular, this implies that
$L^\land\cong L^{op}$ for a~completely distributive lattice $L$.

We can regard the~value of $P(x,y)$ as impossibility of
distribution of total benefit between two players, whose gains are
$x\in L$ and $y\in L^{op}$. For simplicity let $L=[0,1]$, then
the~players want to obtain $x$ and $1-y$ respectively. If
$x+(1-y)>1\iff x>y$, then this is impossible.
\end{example}

For continuous semilattices $S_1$, $S_2$ with zeros consider
separating compatibilities $P_1:S_1\times
\Hat S_1\to\{0,1\}$, $P_2:S_2\times \Hat S_2\to\{0,1\}$. If
$R:S_1\Rightarrow S_2$ is an~ambiguous representation, then
$R^\sms:\Hat S_2\Rightarrow \Hat S_1$ is defined as follows:
$$
R^\sms =
\bigl\{
(\hat y,\hat x)\in \Hat{S_2}\times \Hat{S_1}
\mid
\text{if }xP_1\hat x=1\text{ for some }x\in S_1,
\text{then there is }y\in xR,\; yP_2\hat y=1
\bigr\}.
$$

Recall that each ``negative statement'' $\hat x\in\Hat S_1$ is
completely determined with the~set of all $x\in S_1$ it prohibits,
i.e., with $\{x\in S_1\mid xP_i\hat x=1\}$. Then $\hat x$ can be
considered as attainable from $\hat y\in\Hat S_2$ if and only if it
prohibits only those $x\in S_2$ that have possible consequences
$y\in xR$ incompatible with~$\hat y$ (hence, if $\hat y$ is valid,
then so must be $\hat x$). Passing from $R:S_1\Rightarrow S_2$ to
$R^\sms:\Hat S_2\Rightarrow\Hat S_1$ is analogous to passing from
$A\rightarrow B$ to $\neg B\rightarrow\neg A$ in logic.

For a~completely distributive~\cite{GHK:CLD:03} lattice $L$,
a~pseudo-invertible $L$-fuzzy ambiguous representation $R$ of $S$
in $S'$ is in fact a~pseudo-invertible crisp ambiguous
representation $R:S\Rightarrow S'\times L$ with the~additional
requirement that $xR(y,\alpha_1)$ and $xR(y,\alpha_2)$ imply
$xR(y,\alpha_1\land\alpha_2)$. If $xR(y,\alpha)$, then the~element
$\alpha$ of $L$ describes a~restriction on state of nature under
which $y$ is attainable from $x$. By definition there is the~most
restrictive such~$\alpha$. Then $R^\sms$ has the~same meaning as
above, but depends on the~parameter $\alpha$.

\section*{Concluding remarks and future work}

What we proposed is just a~sketch of application of crisp and
$L$-fuzzy ambiguous representations of continuous semillattices. It
is easy to note that properties of $\Sempr$ are similar to ones of
dialogue category, with $\{0,1\}$ being a~tensorial pole, and the
pair $\Sempral$, $\Semprhl$ looks like a~dialogue
chirality~\cite{Mell:PRIMS:16}. Than taking pseudo-inverse could be
tensorial negation. Unfortunately, it is not the case, e.g.,
$\{0,1\}$ is not exponentiable in $\Sempr$, similarly verification
fails for $\Sempral$ and $\Semprhl$. Nevertheless, the~similarity
is not incidental, and we plan to continue this research to model
games in normal and extended form, winning strategies etc. We are
also going to show why completely distributive lattices arise in
games with imperfect information.

Authors also express deepest gratitude to Prof. Michael Zarichnyi
for valuable ideas.



\begin{thebibliography}{9}
\bibitem{GHK:CLD:03}
Gierz G., Hofmann K.H., Keimel K., Lawson J.D., Mislove M., Scott
D.S., Continuous Lattices and Domains. Encyclopedia of Mathematics
and its Applications, vol.~93, Cambridge University Press (2003)

\bibitem{Jstn:StSp:83} Johnstone P.T., Stone spaces. Cambridge Studies in Advanced
Mathematics, Vol.~3, Cambridge University Press, New York (1983)

\bibitem{ML:CWM:98}
Mac Lane S., Categories for the Working Mathematician. 2nd ed.
Springer, N.Y. (1998)

\bibitem{Mell:PRIMS:16}
Melli\`es P.-A., \emph{Dialogue Categories and Chiralities}, Publ.
Res. Inst. Math. Sci. {\bf 52}(4), 359--412 (2016)

\bibitem{Nyk:CapLat:08} Nykyforchyn O.R.,
\emph{Capacities with values in compact Hausdorff lattices}, Applied
Categorical Structures {\bf 15}(3), 243--257 (2008)

\bibitem{NykMyk:ConMes:10} Nykyforchyn O., Mykytsey O.,
\emph{Conjugate measures on semilattices}, Visnyk Lviv Univ. {\bf 72}, 221--231 (2010)

\bibitem{NykRep:AmbRep:11}
Nykyforchyn O., Repov\v s D., \emph{Ambiguous representations as
fuzzy relations between sets}, Fuzzy Sets and Systems {\bf 173},
25--44 (2011).

\bibitem{Scott:DDS:82}
Scott D.S., Domains for denotational semantics. In: Nielsen M.,
Schmidt E.M. (eds) Automata, Languages and Programming. ICALP 1982.
Lecture Notes in Computer Science, vol 140. Springer, Berlin,
Heidelberg (1982)

\end{thebibliography}

%
%

\end{document}